[1994/12/01]
\documentclass[10pt, reqno]{amsart}
\usepackage[english, activeacute]{babel}
\usepackage{amsmath,amsthm,amsxtra}
\usepackage{epsfig}
\usepackage{amssymb}
\usepackage{latexsym}
\usepackage{amsfonts}
\usepackage{hyperref}
\title{Almost sharp nonlinear scattering in one-dimensional Born-Infeld equations arising in nonlinear Electrodynamics}
\author{Miguel A. Alejo}
\thanks{M. A. was partially funded by Product. CNPq grant no. 305205/2016-1, Universal 16 
CNPq grant no. 431231/2016-8 and MathAmSud/Capes EEQUADD collaboration Math16-01.}
\author{Claudio Mu\~noz}
\thanks{C.M. is partially supported by Fondecyt no. 1150202, Millennium Nucleus Center for Analysis of PDE NC130017, Fondo Basal CMM, and MathAmSud EEQUADD collaboration Math16-01.}
\address{Departamento de Matem\'atica, Universidade Federal de Santa Catarina, Brasil}
\email{miguel.alejo@ufsc.br}
\address{Departamento de Ingenier\'ia Matem\'atica and CMM UMI 2807-CNRS \\ Universidad de Chile, Santiago, Chile}
\email{cmunoz@dim.uchile.cl}
\date{August, 2017}
\subjclass[2000]{Primary 37K15, 35Q53; Secondary 35Q51, 37K10}
\keywords{Born-Infeld equation, scattering, decay estimates, Virial}

\chardef\bslash=`\\ 

\newtheorem{thm}{Theorem}[section]

\newtheorem{lem}[thm]{Lemma}

\theoremstyle{remark}
\newtheorem{rem}{Remark}[section]

\numberwithin{equation}{section}

\newcommand{\R}{\mathbb{R}}

\newcommand{\la}{\lambda}

\newcommand{\sech}{\operatorname{sech}}

\newcommand{\be}{\begin{equation}}
\newcommand{\ee}{\end{equation}}
\newcommand{\bp}{\begin{proof}}
\newcommand{\ep}{\end{proof}}
\newcommand{\bel}{\begin{equation}\label}
\newcommand{\eeq}{\end{equation}}
\newcommand{\bea}{\begin{eqnarray}}
\newcommand{\eea}{\end{eqnarray}}
\newcommand{\bee}{\begin{eqnarray*}}
\newcommand{\eee}{\end{eqnarray*}}
\newcommand{\ben}{\begin{enumerate}}
\newcommand{\een}{\end{enumerate}}

 \providecommand{\abs}[1]{\lvert#1 \rvert}

\newcommand{\ve}{\varepsilon}

\newcommand{\eval}[2][\right]{\relax
  \ifx#1\right\relax \left.\fi#2#1\rvert}

\let\abs=\envert

\begin{document}
\begin{abstract}
We study decay of small solutions of the Born-Infeld equation in 1+1 dimensions, 
a quasilinear scalar field equation modeling nonlinear electromagnetism, as well as 
branes in String theory and minimal surfaces in Minkowski space-times. From the work of 
Whitham, it is well-known that there is no decay because of arbitrary solutions traveling 
to the speed of light just as linear wave equation.  However, even if there is no global decay 
in 1+1 dimensions, we are able to show that all globally small $H^{s+1}\times H^s$, $s>\frac12$ 
solutions do decay to the zero background state in space, inside a strictly proper subset of the 
light cone.  We prove this result by constructing a Virial identity related to a momentum law, in the
spirit of works \cite{KMM,KMM1}, as well as a Lyapunov functional that controls the $\dot H^1 \times L^2$ energy.  
\end{abstract}
\maketitle \markboth{Born-Infeld decay estimates}{Miguel A. Alejo and Claudio Mu\~noz}
\renewcommand{\sectionmark}[1]{}

\section{Introduction and main results}

\medskip

\subsection{The model} This note is concerned with the Born-Infeld \cite{BI} equation (BI) in $\R^{1+1}$:
\be\label{BI} 
\begin{aligned}
 (1- (\partial_t u)^2) \partial_x^2 u +  2\partial_x u\, \partial_t u \, &\partial_{tx}^2 u - (1+(\partial_xu )^2) \partial_t^2 u = 0,\\
 (t,x) \in & \R^{1+1}.
\end{aligned}
\ee

Here and along this paper $u=u(t,x) \in \R$ is a scalar field. BI equation was originally derived by Born and Infeld in 1934 to describe nonlinear 
electrodynamics, a generalization of the standard (linear) Maxwell equations, see \cite{Cher} for a detailed introduction to the field.  
A particular form of an  electromagnetic plane wave solution to the original, three dimensional BI equations, leads to the quasilinear equation \eqref{BI}.

\medskip

Although strictly nonlinear, equation \eqref{BI} shares many similitudes with the classical linear wave equation. 
Given any $C^2$ real-valued profile $\phi=\phi(s)$, both 
\be\label{upm}
u_\pm (t,x) := \phi(x \, \pm \, t)
\ee
are solutions to \eqref{BI} \cite{W}, just as the D'{}Alembert solution is to linear wave. This simple fact reveals that no standard decay estimates are expected in $1+1$ dimensions. Additionally, a form of ``two-soliton'' solution was found by E. Schr\"odinger \cite{Sch}. A Lax-pair representation 
for \eqref{BI}, just as in the case of Euler equations, was found in \cite{BD}. 

\medskip

From a physical point of view, the Born-Infeld equation \eqref{BI} can be derived as \emph{the Euler-Lagrange} 
equation associated to the Lagrangian density \cite{BD}
\be\label{Lag}
\mathcal L[\partial^\mu u] := (1+\partial^\mu u \partial_\mu u)^{1/2} =  (1+ (\partial_x u)^2-(\partial_t u)^2)^{1/2} ,
\ee
whose critical points also represents standard ``minimal surfaces'' in Minkowski geometry. More precisely, 
a simple computation reveals that \eqref{BI} can be written as the minimal surface equation
\be\label{MS}
\partial_t \Bigg(  \frac{ \partial_tu }{(1+ (\partial_x u)^2-(\partial_t u)^2)^{1/2}}\Bigg) -\partial_x \Bigg(  \frac{ \partial_x u }{(1+ (\partial_x u)^2-(\partial_t u)^2)^{1/2}}\Bigg) =0. 
\ee
This equation can be also recast as a mass conservation dynamics, which formally implies that the quantity 
\[
\int  \frac{ \partial_tu }{(1+ (\partial_x u)^2-(\partial_t u)^2)^{1/2}},
\]
is conserved along the dynamics. From this particular point of view, Lindblad \cite{Lindblad} showed global existence 
for small data and decay estimates of order $t^{-(n-1)/2}$ for the equivalent version of \eqref{MS} in $\R^{1+n}$ Minkowski space-time, 
provided $n\geq 1$. Similar results were also proved by Chae and Huh \cite{Chae} in more generality. Both results are strongly 
influenced by vector field techniques \cite{Kl1}, as well as the treatment of key null forms (such as $Q(u,v):= \partial^\mu u \partial_\mu v$), 
present in \eqref{MS}. In the case $n=1$, Lindblad \cite{Lindblad} constructs global small solutions by assuming sufficiently smooth, compactly
supported data. However, as far as we understand, no decay estimate is given in this particular setting ($n=1$), which remains the only case
where no result of this type was available. 

\medskip

Additionally, the study of \eqref{BI} has gained a large impulse in String theory in the last years because gauge fields
on a $D$-brane (that arise from attached open strings) are described by the same type of Lagrangian as \eqref{Lag}; 
see \cite{KZZ} and references therein for more details. The literature on the above subjects is huge; a fairly incomplete set of references on the ``critical'' submanifold problem from different points of view is given by \cite{Brendle,Bre,NS, KL,Don}.  

%

\subsection{Main result} In this paper we consider the remaining case $n=1$, where no decay is present because the existence of arbitrary 
solutions moving to the speed of light. However, we will show that all sufficiently smooth and small solutions to \eqref{BI} must 
decay to zero inside a slightly proper subset of the light cone.
 
\begin{thm}\label{TH1}
Let $C>0$ be an arbitrary constant. Assume that for $\ve>0$ sufficiently small the solution $(u,\partial_tu )(t)$ of \eqref{BI} satisfies 
\be\label{Stab}
\sup_{t\in\R}\|(u,\partial_t u)(t)\|_{(H^{s+1}\times H^s)(\R)} <\ve, \quad s>\frac12.
\ee
Then, given the time-depending interval $I(t):= \Big( \frac{-C |t|}{\log^2 |t|} ,\frac{C |t|}{\log^2 |t|} \Big)$, $|t|\geq 2$, one has
\be\label{AS}
\lim_{t\to \pm\infty}\|(u,\partial_tu )(t) \|_{(\dot H^1\times L^2)(I(t))} =0.
\ee
\end{thm}

Decay estimate \eqref{AS} in Theorem \ref{TH1} is almost sharp, in the sense that solution $u_\pm(t,x)$ given in \eqref{upm} are natural counterexamples to time-decay, with both solutions moving along extremal light cones. The interval $I(t)$ can be slightly improved, see \cite{MPP} for more details, but it seems hard to obtain decay inside the whole light cone $|x| \leq t-a$, for some $a>0$. Complementing these results, strictly outside of the light cone, it is clear that no interesting dynamics occurs (because of finite speed of propagation).

\medskip

Stability condition \eqref{Stab} is ensured for instance by applying Lindblad's second theorem in \cite{Lindblad}; note that it is highly unlikely 
to obtain \eqref{AS} for data in the ``energy space'' $H^1 \times L^2$ only because of the scaling critical regularity 
in dimension $n$ given by $H^{n/2 +1} \times  H^{n/2 }$. Actually, from the proof, and because of the Sobolev embedding,
we will only need \eqref{Stab} to be satisfied in $ \dot H^{s+1} \times \dot H^{s}$, $s>\frac12$. See also Stefanov \cite{Ste} for a detailed account of equation \eqref{BI} posed on Sobolev spaces only; in that case only results for dimensions $n\geq 3$ are available.

\medskip

As a by-product of Theorem \ref{AS},  one can also get a mild rate of decay for the solution. It can be proved (see \eqref{Better_decay}) that for any $c_0>0$,
\[
\int_2^\infty \!\! \int e^{-c_0|x|} ( (\partial_x u)^2 + (\partial_t u)^2) (t,x)dx dt \lesssim_{c_0} \ve^2,
\]
which describes an averaged rate of local $L^2$ decay for the dynamics. 

\medskip

We prove this result by following very recent developments concerning the decay of solutions in  $1+1$ dimensional scalar field models. Kowalczyk, Martel and 
the second author showed in \cite{KMM,KMM1} that well chosen Virial functionals can describe in great generality the
decay mechanism for models where standard scattering is not available (i.e. there is modified scattering), either because the dimension 
is too small, or the nonlinearity is long range. Moreover, this decay mechanism also describes ``nonlinear scattering'', in the sense 
that solutions like \eqref{upm} are also discarded by Theorem \ref{TH1} inside $I(t)$, $t\to+\infty$. Previous Virial-type decay estimates 
were obtained by Martel-Merle and Merle-Rapha\"el \cite{MM,MR} in the case of generalized KdV and nonlinear Schr\"odinger equations. Moreover, 
the results proved in this note and in \cite{KMM,KMM1,MM,MR} apply to equations which have long range nonlinearities, as well as very low or null 
decay rates. In this sense, Theorem \ref{TH1} shows (local) decay even when there is no general decay, and as far as we know, it is a first application 
of these new methods to quasilinear equations. See also \cite{MPP} for another application of this technique to the case of Boussinesq equations, a 
fourth order wave-fluid model. Additionally, unlike other wave-like models \cite{AMP}, no small solitary wave nor breather-like solution persists in 
time inside the proper light cone $I(t)$.

\subsection{Organization of this paper} This note is planned as follows. In Section \ref{2} we present some preliminary lemmas, proving 
a Virial identity (Lemma \ref{derivada}), which we will need in the proof of the Main Theorem. In Section \ref{22} we show that the 
$\dot H^1 \times L^2$ local norm of the solution must be integrable in time; finally Section \ref{3} deals with the end of the proof of Theorem \ref{TH1}.

\bigskip

\section{Integration of the dynamics}\label{2}

\medskip

\subsection{Virial identity}\label{21}

The purpose of this section is to present a new Virial identity  which we will need in the proof of Theorem \ref{TH1}.

\medskip

In what follows, we consider $t\geq 2$ only, and
\be\label{la}
\la(t):= \frac{C t}{\log^2 t}, \qquad \frac{\la'(t)}{\la(t)} = \frac{1}{t} \Big(1 -\frac{2}{\log t} \Big).
\ee

We  introduce a new Virial identity for the Born-Infeld equation \eqref{BI}. 
Indeed, let $\varphi := \tanh$, and let $\mathcal{I}$ be defined as   
\be\label{vir}
\mathcal{I}(t) ~   :=  ~ {} - \int  \varphi \Big( \frac{x}{\la(t)}\Big) \frac{ \partial_tu \,  \partial_x u}{(1+ (\partial_x u)^2-(\partial_t u)^2)^{1/2}}.
\ee
Clearly $\mathcal{I}(t)$ is well defined as long as \eqref{Stab} is satisfied. Here we use the fact that both 
$\partial_x u$ and $\partial_t u$ are small in $L^\infty$ thanks to the Sobolev embedding. Moreover.
\[
\sup_{t\in \R} |\mathcal{I}(t)| \lesssim \ve^2. 
\]
Note additionally that the denominator in \eqref{vir} contains the gradient of $u$ contracted with the standard Minkowski metric. A time-dependent weight was also considered in \cite{MMT}, but with different goals.

\medskip

The choice of $\mathcal I(t)$ is partly motivated by the energy identity that one obtains in \eqref{phiEt}. 

\begin{lem}[Virial identity]\label{derivada}
We have
\be\label{dI}
\begin{aligned}
\frac{d}{dt} \mathcal{I} (t)  = &~  \frac{1}{\la(t)} \int \varphi' \Big( \frac{x}{\la(t)}\Big)  \Bigg( 1 -\frac{1 - (\partial_t u)^2}{(1+ (\partial_x u)^2-(\partial_t u)^2)^{1/2}}  \Bigg)\\
& ~ {}  + \frac{\la'(t)}{\la(t)} \int \frac{x}{\la(t)}  \varphi' \Big( \frac{x}{\la(t)}\Big) \frac{ \partial_tu \,  \partial_x u}{(1+ (\partial_x u)^2-(\partial_t u)^2)^{1/2}}. 
\end{aligned}
\ee
\end{lem}
\begin{proof}
We readily have
\be\label{vt}
\begin{aligned}
 \frac{d}{dt}\mathcal{I}(t) = &~ {} -\int \varphi \Big( \frac{x}{\la(t)}\Big) Q(\partial_t u, \partial_x u, \partial_{t}^2 u, \partial_{tx}^2 u) \\
 & +  \frac{\la'(t)}{\la(t)} \int \frac{x}{\la(t)}  \varphi' \Big( \frac{x}{\la(t)}\Big) \frac{ \partial_tu \,  \partial_x u}{(1+ (\partial_x u)^2-(\partial_t u)^2)^{1/2}},
\end{aligned}
\ee
where 
\be\label{Q}
\begin{aligned}
& Q(\partial_t u, \partial_x u, \partial_{t}^2 u, \partial_{tx}^2 u):= \\
& \quad= \frac{ (\partial_{tx}^2 u \, \partial_t u +  \partial_x u\, \partial_t^2 u)(1+ (\partial_x u)^2-(\partial_t u)^2)-(\partial_x u\, \partial_{tx}^2 u  - \partial_t u \,\partial_t^2 u )\partial_x u \,\partial_t u  }{(1+ (\partial_x u)^2-(\partial_t u)^2)^{3/2}}.
\end{aligned}
\ee
The second term in \eqref{vt} is already present in \eqref{dI}, so we focus only on the first term leading to the numerator $Q(\partial_t u, \partial_x u, \partial_{t}^2 u, \partial_{tx}^2 u)$. We rewrite it in the following compact form:
\be\label{Qnum}
\begin{aligned}
& (\partial_{tx}^2 u\, \partial_t u +  \partial_x u \,\partial_t^2 u)(1+ (\partial_x u)^2-(\partial_t u)^2)-(\partial_x u \,\partial_{tx}^2 u  - \partial_t u \partial_t^2 u )\partial_x u\, \partial_t u  =\\
& \quad  = -(1+(\partial_x u)^2-(\partial_t u)^2) \partial_x(1-(\partial_t u)^2)  \\
& \quad \quad + (1-(\partial_t u)^2)\partial_x\Big(\frac{1}{2}(1+(\partial_x u)^2-(\partial_t u)^2)\Big).
 \end{aligned}
\ee
We assume this identity, see below for the proof. Hence, using \eqref{Qnum} we simplify $\frac{d}{dt}\mathcal I(t)$ as follows: 
\be\label{Pt2}
\begin{aligned}
\frac{d}{dt}\mathcal I(t) = &~ {} \int \varphi \Big( \frac{x}{\la(t)}\Big) \frac{(1+(\partial_x u)^2-(\partial_t u)^2) \partial_x(1-(\partial_t u)^2)  }
{(1+(\partial_x u)^2-(\partial_t u)^2)^{3/2}} \\
& {}-\int \varphi \Big( \frac{x}{\la(t)}\Big) \frac{(1-(\partial_t u)^2)\partial_x\Big(\frac{1}{2}(1+(\partial_x u)^2-(\partial_t u)^2)\Big)}
{(1+(\partial_x u)^2-(\partial_t u)^2)^{3/2}}\\
& +  \frac{\la'(t)}{\la(t)} \int_\R \frac{x}{\la(t)}  \varphi' \Big( \frac{x}{\la(t)}\Big) \frac{ \partial_tu \, 
\partial_x u}{(1+ (\partial_x u)^2-(\partial_t u)^2)^{1/2}} \\
=& ~ A_1+A_2+B. 
\end{aligned}
\ee
The important term above is $A_1+A_2$.  We have 
\[
\begin{aligned}
A_1+A_2& = -\int \varphi \Big( \frac{x}{\la(t)}\Big) \frac{\partial_x((\partial_t u)^2)}{(1+ (\partial_x u)^2-(\partial_t u)^2)^{1/2}} \\
& \quad - \frac{1}{2}\int \varphi \Big( \frac{x}{\la(t)}\Big) (1-(\partial_t u)^2) \frac{\partial_x(1+(\partial_x u)^2-(\partial_t u)^2)}{(1+(\partial_x u)^2-(\partial_t u)^2)^{3/2}}\\
&= -\int \varphi  \Big( \frac{x}{\la(t)}\Big)\frac{\partial_x((\partial_t u)^2)}{(1+ (\partial_x u)^2-(\partial_t u)^2)^{1/2}}  \\
& \quad + \int \varphi \Big( \frac{x}{\la(t)}\Big)(1-(\partial_t u)^2)\partial_x\Bigg( \frac{1}{(1+ (\partial_x u)^2-(\partial_t u)^2)^{1/2}}-1\Bigg).\\
\end{aligned}
\]

\medskip

Now integrating by parts the second integral we get
\[
\begin{aligned}
A_1+A_2& = -\int \varphi \Big( \frac{x}{\la(t)}\Big) \frac{\partial_x((\partial_t u)^2)}{(1+ (\partial_x u)^2-(\partial_t u)^2)^{1/2}} \\
&\quad - \frac1{\la(t)}\int \varphi' \Big( \frac{x}{\la(t)}\Big) (1-(\partial_t u)^2)\Bigg( \frac{1}{(1+ (\partial_x u)^2-(\partial_t u)^2)^{1/2}}-1\Bigg) \\
& \quad + \int \varphi \Big( \frac{x}{\la(t)}\Big) \partial_x((\partial_t u)^2)\Bigg( \frac{1}{(1+ (\partial_x u)^2-(\partial_t u)^2)^{1/2}}-1\Bigg).
\end{aligned}
\]
We continue simplifying,
\[
\begin{aligned}
A_1+A_2&= - \frac1{\la(t)}\int \varphi' \Big( \frac{x}{\la(t)}\Big) (1-(\partial_t u)^2)\Bigg( \frac{1}{(1+ (\partial_x u)^2-(\partial_t u)^2)^{1/2}}-1\Bigg)\\
&\quad  - \int \varphi \Big( \frac{x}{\la(t)}\Big) \partial_x((\partial_t u)^2).
\end{aligned}
\]
Finally, integrating again by parts the second integral, we obtain
\[
\begin{aligned}
A_1+A_2 &= -\frac1{\la(t)} \int \varphi'\Big( \frac{x}{\la(t)}\Big) \Bigg((1-(\partial_t u)^2)\Bigg( \frac{1}{(1+ (\partial_x u)^2-(\partial_t u)^2)^{1/2}}-1\Bigg) -(\partial_t u)^2\Bigg)\\
&=  \frac1{\la(t)} \int \varphi'\Big( \frac{x}{\la(t)}\Big)\Bigg(1 - \frac{1-(\partial_t u)^2}{(1+ (\partial_x u)^2-(\partial_t u)^2)^{1/2}} \Bigg).
\end{aligned}
\]
Collecting $A_1+A_2+B$ we get the desired result.
\end{proof}

\medskip
%
%
%

\begin{proof}[Proof of \eqref{Qnum}] 
We have
\[
\begin{aligned}
&(\partial_{xt}^2u\partial_tu + \partial_xu\partial_{t}^2u)(1+(\partial_x u)^2-(\partial_t u)^2)-
(\partial_x u\partial_{xt}^2u-\partial_tu\partial_{t}^2u)\partial_x u\partial_t u =\\
&\quad = \partial_{xt}^2u\partial_t u(1+(\partial_x u)^2-(\partial_t u)^2) + \partial_xu \partial_{t}^2 u(1+(\partial_x u)^2-(\partial_t u)^2) \\
&\quad  \quad - \partial_x u\partial_t u(\partial_x u \partial_{xt}^2 u-\partial_t u\partial_{t}^2 u)\\
&\quad = \partial_{xt}^2 u\partial_t u(1+(\partial_x u)^2-(\partial_t u)^2) + \partial_x u\partial_{t}^2 u(1+(\partial_x u)^2) 
-  \partial_x u\partial_{t}^2 u(\partial_t u)^2 \\
&\quad  \quad - (\partial_x u)^2\partial_t u\partial_{xt}^2 u + \partial_x u(\partial_t u)^2\partial_{t}^2 u\\
&\quad = \partial_{xt}^2 u\partial_t u(1+(\partial_x u)^2-(\partial_t u)^2) + \partial_x u((1-(\partial_t u)^2)\partial_{x}^2 u + 
2\partial_x u\partial_t u\partial_{x}^2 u) \\
& \quad \quad - (\partial_x u)^2\partial_t u\partial_{xt}^2 u\\
&\quad = \partial_{xt}^2 u\partial_t u(1+(\partial_x u)^2-(\partial_t u)^2) + (1-(\partial_t u)^2)\partial_x u\partial_{x}^2 u
+ 2(\partial_x u)^2\partial_t u\partial_{xt}^2 u \\
&\quad  \quad - (\partial_x u)^2\partial_t u\partial_{xt}^2 u\\
&\quad = (\partial_x u)^2(2\partial_t u\partial_{xt}^2 u) + \partial_t u\partial_{xt}^2 u(1-(\partial_t u)^2) + (1-(\partial_t u)^2)\partial_x u\partial_{x}^2 u\\
&\quad = (\partial_x u)^2(2\partial_t u\partial_{xt}^2 u) + (2\partial_t u\partial_{xt}^2 u)(1-(\partial_t u)^2) \\
&\quad \quad + (1-(\partial_t u)^2)(\partial_x u\partial_{x}^2 u-\partial_t u\partial_{xt}^2 u)\\
&\quad = (1+(\partial_x u)^2-(\partial_t u)^2)(2\partial_t u\partial_{xt}^2 u) + (1-(\partial_t u)^2)(\partial_x u\partial_{x}^2 u-\partial_t u\partial_{xt}^2 u)\\
&\quad = -(1+(\partial_x u)^2-(\partial_t u)^2)\partial_x(1-(\partial_t u)^2) \\
& \quad \quad + (1-(\partial_t u)^2)\partial_x\Big(\frac{1}{2}(1+(\partial_x u)^2-(\partial_t u)^2)\Big).
\end{aligned}
\]
\end{proof}

\bigskip

\subsection{Energy identity} Let consider $\phi := \varphi'^2 =\sech^4$. The following Lemma will be useful
to prove the integral estimate \eqref{integral} in  Section \ref{22}.

\begin{lem}\label{phiE} Let $(u,\partial_tu)$ be a global small solution of \eqref{BI}. Then, 
\be\label{phiEt}
\begin{aligned}
&\frac{d}{dt}\int \phi\Big(\frac{x}{\lambda(t)} \Big)\Bigg( \frac{1 + (\partial_x u)^2}{(1+ (\partial_x u)^2-(\partial_t u)^2)^{1/2}} -1\Bigg) = \\
& =  -\frac{1}{\lambda(t)} \int \phi' \Big(\frac{x}{\lambda(t)} \Big) \frac{ \partial_x u \, \partial_t u }{(1+ (\partial_x u)^2-(\partial_t u)^2)^{1/2}} \\
&  \quad - \frac{\lambda'(t)}{\lambda(t)} \int \frac{x}{\lambda(t)}\phi' \Big(\frac{x}{\lambda(t)}\Big)\Bigg( \frac{1 + (\partial_x u)^2}{(1+ (\partial_x u)^2-(\partial_t u)^2)^{1/2}} -1\Bigg). 
\end{aligned}\ee
\end{lem}

\begin{proof}
We have 
\[
\begin{aligned}
 &\frac{d}{dt}\int \phi\Big(\frac{x}{\lambda(t)} \Big)\Bigg( \frac{1 + (\partial_x u)^2}{(1+ (\partial_x u)^2-(\partial_t u)^2)^{1/2}} -1\Bigg)\\
 &\quad = \int \partial_t\Big( \phi \Big(\frac{x}{\lambda(t)} \Big) \Big)\Bigg( \frac{1 + (\partial_x u)^2}{(1+ (\partial_x u)^2-(\partial_t u)^2)^{1/2}} -1\Bigg) \\
 &\quad \quad +\int \phi \Big(\frac{x}{\lambda(t)} \Big)\partial_t\Bigg( \frac{1 + (\partial_x u)^2}{(1+ (\partial_x u)^2-(\partial_t u)^2)^{1/2}} -1\Bigg) \\
 & \quad =: ~ E_1 + E_2 .
\end{aligned}
\]
As for $E_1$, we easily have
\[
\begin{aligned}
& \int \partial_t\Big( \phi \Big(\frac{x}{\lambda(t)} \Big) \Big)\Bigg( \frac{1 + (\partial_x u)^2}{(1+ (\partial_x u)^2-(\partial_t u)^2)^{1/2}} -1\Bigg) = \\
& \quad = - \frac{\lambda'(t)}{\lambda(t)} \int \frac{x}{\lambda(t)}\phi' \Big(\frac{x}{\lambda(t)}\Big)\Bigg( \frac{1 + (\partial_x u)^2}{(1+ (\partial_x u)^2-(\partial_t u)^2)^{1/2}} -1\Bigg),
\end{aligned}
\]
which is the second term in the right hand side of \eqref{phiEt}.

\medskip

Now, for $E_2$ we have
\be\label{E_2}
\begin{aligned}
& E_2= \int \phi \Big(\frac{x}{\lambda(t)} \Big)\partial_t\Bigg( \frac{1 + (\partial_x u)^2}{(1+ (\partial_x u)^2-(\partial_t u)^2)^{1/2}} -1\Bigg) \\
& \quad = \int  \phi \Big(\frac{x}{\lambda(t)} \Big) \widetilde Q(\partial_x u, \partial_t u, \partial_t^2 u, \partial_{tx}^2 u),
\end{aligned}
\ee
where 
\be\label{tQ}
\begin{aligned}
&  \widetilde Q(\partial_x u, \partial_t u, \partial_t^2 u, \partial_{tx}^2 u) := \\
& \quad = \Bigg( \frac{2\partial_{xt}^2  u\, \partial_x u(1+ (\partial_x u)^2-(\partial_t u)^2)-(1+(\partial_x u)^2)(\partial_x  u\, \partial_{xt}^2 u-\partial_t u\, \partial_{t}^2 u)}{(1+ (\partial_x u)^2-(\partial_t u)^2)^{3/2}} \Bigg).
\end{aligned}
\ee
Similarly to \eqref{Q}, using \eqref{BI} we have
\[
\begin{aligned}
&2\partial_{xt}^2 u \, \partial_x u(1+ (\partial_x u)^2 - (\partial_t u)^2)-(1+(\partial_x u)^2)(\partial_x u \, \partial_{xt}^2 u-\partial_t u \, \partial_{t}^2 u)\\
&= 2\partial_{xt}^2 u\, \partial_x u(1+ (\partial_x u)^2) - 2\partial_{xt}^2 u\, \partial_x u(\partial_{t} u)^2 \\
& \quad + \partial_t u(1+(\partial_x u)^2)\partial_{t}^2 u 
\quad - (1+(\partial_x u)^2)\partial_x u\, \partial_{xt}^2 u\\
&= 2\partial_{xt}^2 u \, \partial_x u(1+ (\partial_x u)^2) - 2\partial_{xt}^2 u\, \partial_x u(\partial_{t} u)^2 \\
& \quad  + \partial_t u\big((1-(\partial_t u)^2)\partial_x^2 u+2\partial_x u\, \partial_t u\, \partial_{xt}^2 u\big) - (1+(\partial_x u)^2)\partial_x u\, \partial_{xt}^2 u\\
&= \partial_{xt}^2 u\, \partial_x u(1+ (\partial_x u)^2) + \partial_t u(1-(\partial_t u)^2)\partial_x^2 u\\
&= \partial_{xt}^2 u\, \partial_x u(1+ (\partial_x u)^2-(\partial_t u)^2) + \partial_x u\, \partial_x^2 u\, (\partial_t u)^2 \\
& \quad + \partial_t u \, (1+(\partial_x u)^2-(\partial_t u)^2)\partial_x^2 u -\partial_t u(\partial_x u)^2\partial_x^2 u\\
&= (\partial_{xt}^2 u\, \partial_x u+\partial_t u\, \partial_x^2 u)(1+(\partial_x u)^2-(\partial_t u)^2) 
+ \partial_x u\, \partial_t u(\partial_{xt}^2 u\, \partial_t u-\partial_x u\, \partial_x^2 u)\\
&=\partial_x(\partial_t u\, \partial_x u) (1+(\partial_x u)^2-(\partial_t u)^2) 
- \frac12\partial_x u\, \partial_t u\, \partial_x(1+(\partial_x u)^2-(\partial_t u)^2) .
\end{aligned}
\]
Replacing this last identity in $E_2$ in \eqref{E_2}, 
\[
\begin{aligned}
 E_2 =& ~ \int \phi\Big(\frac{x}{\lambda(t)} \Big)\Bigg(\frac{\partial_x(\partial_t u\, \partial_x u)}{(1+ (\partial_x u)^2-(\partial_t u)^2)^{1/2}} \\
  & \qquad \qquad \qquad \quad+ \partial_x u\, \partial_t u\,\partial_x\Big(\frac{1}{(1+ (\partial_x u)^2-(\partial_t u)^2)^{1/2}} \Big)\Bigg).
\end{aligned}
\]  
Integrating by parts,
\[
\begin{aligned}
E_2  = & {} -\frac{1}{\lambda(t)} \int \phi' \Big(\frac{x}{\lambda(t)} \Big) \frac{ \partial_x u \, \partial_t u }{(1+ (\partial_x u)^2-(\partial_t u)^2)^{1/2}}, 
 \end{aligned}
\]
as desired.
\end{proof}

\subsection{Integration of the dynamics}\label{22}
\medskip

The purpose of this subsection is to show the following integral estimate, which  gives us 
a proper control of the dynamics in time of the norm of the solution of \eqref{BI}:

\begin{lem}\label{integral} 
For $\la(t)$ given as in \eqref{la}, and $(u,\partial_t u)(t)$ satisfying \eqref{Stab}, we have the averaged estimate 
\be\label{Integra}
 \int_2^\infty  \frac{1}{\la(t)} \int \sech^2 \Big( \frac{x}{\la(t)}\Big)  ( (\partial_t u)^2 + (\partial_x u)^2)(t,x) dxdt \lesssim \ve^2.
\ee
Moreover, there exists an increasing sequence of times $t_n\to +\infty$ such that 
\be\label{seq}
\lim_{n\to +\infty} \int \sech^2 \Big( \frac{x}{\la(t_n)}\Big)  ( (\partial_t u)^2 + (\partial_x u)^2) (t_n,x) dx =0.
\ee
\end{lem}

This result roughly states that, all global small solutions must decay locally in space. The exact region of decay is determined by the scaling parameter $\la(t)$, chosen in \eqref{la}.

\medskip

In order to show Lemma \ref{integral}, we use the new Virial identity for \eqref{vir} presented for the Born-Infeld equation \eqref{BI}.





\begin{proof}
From \eqref{dI}, we have the identity
\[
\begin{aligned}
\frac{d}{dt} \mathcal{I}(t)  = &~  \frac{1}{ \la(t)} \int \varphi' \Big( \frac{x}{\la(t)}\Big)  \Bigg( 1 -\frac{1 - (\partial_t u)^2}{(1+ (\partial_x u)^2-(\partial_t u)^2)^{1/2}}  \Bigg)\\
& ~ {}  + \frac{\la'(t)}{\la(t)} \int \frac{x}{\la(t)}  \varphi' \Big( \frac{x}{\la(t)}\Big) \frac{ \partial_tu \,  \partial_x u}{(1+ (\partial_x u)^2-(\partial_t u)^2)^{1/2}} \\
= : & ~   \mathcal{I}_1 +\mathcal{I}_2.
\end{aligned}
\]
We will estimate each term above separately. First, we claim
\be\label{Estimate_K}
\abs{ 1 -\frac{1 - (\partial_t u)^2}{(1+ (\partial_x u)^2-(\partial_t u)^2)^{1/2}}  -\frac12 ((\partial_x u)^2+(\partial_t u)^2)  } \leq  (\partial_x u)^4+(\partial_t u)^4.
\ee
This is a simple consequence of the inequality
\[
\abs{ \frac{1 - a }{(1+ b-a)^{1/2}} -1 + \frac12 (a+b)  } \leq  a^2+b^2,
\]
valid for all $a,b\geq 0$ sufficiently small. Consequently,
\[
\begin{aligned}
\mathcal I_1 \geq  &~  \frac{1}{ 2\la(t)} \int \varphi' \Big( \frac{x}{\la(t)}\Big)((\partial_x u)^2+(\partial_t u)^2) \\
& {} - \frac{1}{ \la(t)} \int \varphi' \Big( \frac{x}{\la(t)}\Big)((\partial_x u)^4+(\partial_t u)^4) \\
\geq &~\frac1{4\la(t)} \int \varphi' \Big( \frac{x}{\la(t)}\Big)((\partial_x u)^2+(\partial_t u)^2), \qquad \hbox{(see \eqref{Stab}).}
\end{aligned}
\]
On the other hand, from the estimate
\be\label{MaMa}
(1+ (\partial_x u)^2-(\partial_t u)^2)^{1/2} \geq \frac12,
\ee
and the value of $\la(t)$ in \eqref{la}, we have
\[
\begin{aligned}
\abs{\mathcal I_2} \leq &~   \frac 2t \int \frac{|x|}{\la(t)}  \varphi' \Big( \frac{x}{\la(t)}\Big)|\partial_tu ||\partial_x u| \\
 \leq  & ~ \frac {C\la(t)}{t^2}  \sup_{x\in\R} \Big(\frac{|x|^2}{\la^2(t)}  \varphi' \Big( \frac{x}{\la(t)}\Big) \Big) \int (\partial_tu)^2 + \frac 1{8\la(t)} \int \varphi' \Big( \frac{x}{\la(t)}\Big)(\partial_x u)^2 \\
 \leq  & ~ \frac {C\ve^2}{t \log^2 t}   + \frac 1{8\la(t)} \int \varphi' \Big( \frac{x}{\la(t)}\Big)(\partial_x u)^2. 
\end{aligned}
\]
We collect the estimates on $\mathcal I_1$ and $\mathcal I_2$ to obtain
\[
\frac{d}{dt} \mathcal{I}(t) \geq \frac1{8\la(t)} \int \varphi' \Big( \frac{x}{\la(t)}\Big)((\partial_x u)^2+(\partial_t u)^2) -\frac {C\ve^2}{t \log^2 t}.
\]
After integration in time we get \eqref{Integra}. Finally, \eqref{seq} is obtained from \eqref{Integra} and the fact that $\la^{-1}(t)$ is not integrable in $[2,\infty)$.
\end{proof}


\bigskip

\subsection{A final estimate} We finish this Section with a small result.

\begin{lem}
Under the conclusions of Lemma \ref{integral}, we have
\be\label{Conclu_n}
\lim_{n\to +\infty} \int \sech^4\Big(\frac{x}{\lambda(t_n)} \Big)\Big( \frac{1 + (\partial_x u)^2}{(1+ (\partial_x u)^2-(\partial_t u)^2)^{1/2}} -1\Big)(t_n) =0.
\ee
\end{lem}

\begin{proof}
Proceeding as in the proof of estimate \eqref{Estimate_K}, we have
\be\label{Estimate_E}
\abs{ \frac{1 + (\partial_x u)^2}{(1+ (\partial_x u)^2-(\partial_t u)^2)^{1/2}} -1 -\frac12 ((\partial_x u)^2+(\partial_t u)^2)  } \leq  (\partial_x u)^4+(\partial_t u)^4.
\ee
Using \eqref{Stab}, we are lead to the estimate
\[
\begin{aligned}
& \int \sech^4\Big(\frac{x}{\lambda(t_n)} \Big)\Big( \frac{1 + (\partial_x u)^2}{(1+ (\partial_x u)^2-(\partial_t u)^2)^{1/2}} -1\Big)(t_n) \\
& \quad \lesssim \int  \sech^2  \Big( \frac{x}{\la(t_n)}\Big)((\partial_x u)^2+(\partial_t u)^2)(t_n),
\end{aligned}
\]
which shows \eqref{Conclu_n}.
\end{proof}

\begin{rem}
Lemma \ref{integral} can also be obtained using a new and different functional. Consider now the slightly modified Virial term
\be\label{J}
\mathcal{J}(t) ~   :=  ~ {} - \int  \varphi \Big( \frac{x}{\la(t)}\Big) \frac{ \partial_tu \,  \partial_x u}{1+ (\partial_x u)^2}.
\ee
The main difference with respect to \eqref{vir} is that $\mathcal J(t)$ is defined in the larger space $(u,\partial_t u)\in \dot H^1 \times L^2$, 
because the denominator is never zero. It is not difficult to show, using the ideas in \eqref{dI}, that the following cleaner identity holds
\[
\begin{aligned}
\frac{d}{dt} \mathcal{J}(t)  = &~  \frac{1}{2\la(t)} \int \varphi' \Big( \frac{x}{\la(t)}\Big) \Bigg( \frac{(\partial_x u)^2 + (\partial_t u)^2}{1+ 
(\partial_x u)^2} \Bigg)\\
&  + \frac{\la'(t)}{\la(t)} \int \frac{x}{\la(t)}  \varphi' \Big( \frac{x}{\la(t)}\Big) \frac{ \partial_tu \,  \partial_x u}{1+ (\partial_x u)^2}, 
\end{aligned}
\]
which reveals a slightly better control of the local $L^2$ norm of  the pair $(\partial_x u,\partial_t u)$. 
However, unless we ensure that $\partial_x u$ is small in certain uniform sense, this last identity cannot be used to prove Lemma \ref{integral} 
for data in $\dot H^1 \times L^2$ only. This is certainly another key signature of quasilinear models, as the BI equation \eqref{BI}.
\end{rem}

\medskip

\begin{rem}
Finally, note that by choosing $\la(t) =\la_0>0$ fixed in \eqref{J}, we have the better integral estimate
\be\label{Better_decay}
\int_2^\infty \!\! \int \sech^2 \Big(\frac{x}{\la_0}\Big) ( (\partial_x u)^2 + (\partial_t u)^2) (t,x)dx dt \lesssim \la_0 \ve^2.
\ee
We can recast this integral as an averaged decay estimate for the local $L^2$-norm of $(\partial_x u,\partial_t u)$, meaning that formally, 
and locally in space, both decay better than $1/\sqrt{t}$. However, note also that \eqref{Better_decay} describes nonlinear scattering, in the
sense that even nonlinear objects, solutions to \eqref{BI}, depart far away from every compact subset of space.
\end{rem}

\section{Proof of Theorem \ref{TH1}}\label{3}

\medskip

We start with the following energy estimate.

\begin{lem}\label{3p1}
Let $(u,\partial_t u)$ be an $H^2 \times H^1$ solution to \eqref{BI} satisfying \eqref{Stab}. Then we have
\be\label{Energy_estimate}
\begin{aligned}
& \abs{\frac{d}{dt}\int \phi\Big(\frac{x}{\lambda(t)} \Big)\Bigg( \frac{1 + (\partial_x u)^2}{(1+ (\partial_x u)^2-(\partial_t u)^2)^{1/2}} -1\Bigg)}  \\
& \qquad \lesssim  \frac1{\la(t)} \int \sech^2 \Big(\frac{x}{\lambda(t)} \Big) ((\partial_x u)^2 +(\partial_t u)^2).
\end{aligned}
\ee
\end{lem}

\begin{rem}
The choice of functional in \eqref{Energy_estimate} is motivated by the ``energy''
\[
\int \Bigg( \frac{1+(\partial_x u)^2}{(1+ (\partial_x u)^2-(\partial_t u)^2)^{1/2}} -1\Bigg),
\]
which is formally conserved by the dynamics. This energy controls the $\dot H^1 \times L^2$ norm if $(\partial_x u)^2-(\partial_t u)^2$ is small, in the sense that 
\[
\int \Bigg( \frac{1+(\partial_x u)^2}{(1+ (\partial_x u)^2-(\partial_t u)^2)^{1/2}} -1\Bigg) \sim \int ((\partial_x u)^2 +(\partial_t u)^2) ;
\]
see more details in \eqref{Lower} below.
\end{rem}

\begin{proof}[Proof of Lemma \ref{3p1}]
Using Lemma \ref{phiE}, we show estimate  \eqref{Energy_estimate} above. First of all, using \eqref{MaMa} we have
\[
\begin{aligned}
& \abs{\frac{1}{\lambda(t)} \int \phi' \Big(\frac{x}{\lambda(t)} \Big) \frac{ \partial_x u \, \partial_t u }{(1+ (\partial_x u)^2-(\partial_t u)^2)^{1/2}} } \\
& \qquad  \lesssim \frac1{\la(t)} \int \sech^2 \Big(\frac{x}{\lambda(t)} \Big) ((\partial_x u)^2 +(\partial_t u)^2).
\end{aligned}
\]
Finally, using \eqref{Estimate_E}, \eqref{la}, and the fact that $\phi=\sech^4$ satisfies the estimate 
$\abs{\frac{x}{\lambda(t)}\phi' (\frac{x}{\lambda(t)}) } \lesssim \sech^2 (\frac{x}{\lambda(t)})$,
\[
\begin{aligned}
& \abs{ \frac{\lambda'(t)}{\lambda(t)} \int \frac{x}{\lambda(t)}\phi' \Big(\frac{x}{\lambda(t)}\Big)\Bigg( \frac{1 + (\partial_x u)^2}{(1+ (\partial_x u)^2-(\partial_t u)^2)^{1/2}} -1\Bigg)} \\
& \qquad \lesssim \frac1t \int  \sech^2\Big(\frac{x}{\lambda(t)}\Big) ((\partial_x u)^2 +(\partial_t u)^2) \\
& \qquad \lesssim \frac1{\la(t)} \int \sech^2 \Big(\frac{x}{\lambda(t)} \Big) ((\partial_x u)^2 +(\partial_t u)^2). 
\end{aligned}
\]
Collecting these two estimates we get \eqref{Energy_estimate}.

\end{proof}

\medskip

We finish the proof of Theorem \ref{TH1}. 

\medskip

\begin{proof}[Proof of Theorem \ref{TH1}] We have from \eqref{Energy_estimate}, and for $t<t_n$,
\[
\begin{aligned}
&  \Bigg| \int \phi\Big(\frac{x}{\lambda(t_n)} \Big)\Bigg( \frac{1 + (\partial_x u)^2}{(1+ (\partial_x u)^2-(\partial_t u)^2)^{1/2}} -1\Bigg)(t_n)  \\
& \qquad - \int \phi\Big(\frac{x}{\lambda(t)} \Big)\Bigg( \frac{1 + (\partial_x u)^2}{(1+ (\partial_x u)^2-(\partial_t u)^2)^{1/2}} -1\Bigg)(t)  \Bigg|  \\
& \qquad \qquad \lesssim \int_{t}^{t_n} \frac1{\la(s)} \int \sech^2 \Big(\frac{x}{\lambda(s)} \Big) ((\partial_x u)^2 +(\partial_t u)^2)dxds.
\end{aligned}
\]
Sending $n$ to infinity, and using \eqref{Conclu_n}, we get
\[
\begin{aligned}
& \abs{\int \phi\Big(\frac{x}{\lambda(t)} \Big)\Big( \frac{1 + (\partial_x u)^2}{(1+ (\partial_x u)^2-(\partial_t u)^2)^{1/2}} -1\Big)(t)  } \\
& \qquad \lesssim \int_{t}^{\infty} \frac1{\la(s)} \int \sech^2 \Big(\frac{x}{\lambda(s)} \Big) ((\partial_x u)^2 +(\partial_t u)^2)dxds,
\end{aligned}
\]
which implies, thanks to Lemma \ref{integral},
\[
\lim_{t\to +\infty}\abs{\int \phi\Big(\frac{x}{\lambda(t)} \Big)\Bigg( \frac{1 + (\partial_x u)^2}{(1+ (\partial_x u)^2-(\partial_t u)^2)^{1/2}} -1\Bigg)(t)  } =0.
\]
Finally, from the inequality \eqref{Estimate_E} we get the lower bound
\be\label{Lower}
\begin{aligned}
& \int \phi\Big(\frac{x}{\lambda(t)} \Big) ((\partial_x u)^2 +(\partial_t u)^2)(t)\\
& \qquad  \lesssim \int \phi\Big(\frac{x}{\lambda(t)} \Big)\Bigg( \frac{1 + (\partial_x u)^2}{(1+ (\partial_x u)^2-(\partial_t u)^2)^{1/2}} -1\Bigg)(t),
\end{aligned}
\ee
which finally shows the validity of Theorem \ref{TH1}.
\end{proof}

\begin{rem}[Final remarks]
It is expected that some of these results hold for larger dimensions, with probably difficult proofs. 
However, in view of the already well-known results \cite{Lindblad} for dimensions $n\geq 2$, we believe 
that the main new contribution in this note is given in the treatment of the scattering problem in the one dimensional case.
\end{rem}


\bigskip
\bigskip

\begin{thebibliography}{99}


\bibitem{AMP} M.A. Alejo, C. Mu\~noz, and J. M. Palacios, \emph{On the variational structure of breather solutions I: Sine-Gordon equation.} J. Math. Anal. Appl. 453 (2017), no. 2, 1111--1138.



\bibitem{AO} M. Arik, F. Neyzi, Y. Nutku, P. J. Olver and J. M. Verosky, \emph{Multi-Hamiltonian structure of the Born-Infeld equation},
J. 	Math. Phys., 30, 1338 (1989).

\bibitem{BI} Born, M. and Infeld, L. 1934. \emph{Foundation of the new field theory}, Proc. Roy. Soc. A 144: 425--451.

\bibitem{Brendle}  Brendle, Simon, \emph{Hypersurfaces in Minkowski space with vanishing mean curvature}. Comm. Pure Appl. Math. 55 (2002), no. 10, 1249--1279.

\bibitem{Bre} Brenier, Y., \emph{Hydrodynamic structure of the augmented Born-Infeld equations}, Arch. Rational Mech. Anal. 172 (2004), 65--91.

\bibitem{BD} J. C. Brunelli and A. Das, \emph{A Lax Representation for the Born-Infeld Equation}. Phys. Lett. B 426, 57 (1998).

\bibitem{Chae} Chae, D., Huh, H.,  \emph{Global existence for small initial data in the Born-Infeld equations}, J.
Math. Phys., 44 (2003), n. 12, 6132--6139.

\bibitem{Cher} A. A. Chernitskii, \emph{Born-Infeld equations}, Encyclopedia of Nonlinear Science, ed. Alwyn Scott. New York and London: Routledge, 2004, pp. 67--69, arXiv:0509087v1 (hep-th).

\bibitem{Don}  Donninger, Roland; Krieger, Joachim; Szeftel, J\'er\'emie; Wong, Willie, \emph{Codimension one stability of the catenoid under the vanishing mean curvature flow in Minkowski space}. Duke Math. J. 165 (2016), no. 4, 723--791.

\bibitem{Kl1} S. Klainerman, \emph{The null condition and global existence to nonlinear wave equations,} Lectures in Applied Mathematics, Vol. 23, Amer. Math. Soc., Providence, RI, pp. 293--326.

\bibitem{KZZ} De-Xing Kong, Qiang Zhang,  and Qing Zhou, \emph{The Dynamics of Relativistic Strings Moving in the Minkowski Space $\R^{1+n}$},  Commun. Math. Phys. 269, 153--174 (2007).


\bibitem{KMM} M. Kowalczyk, Y. Martel, and C. Mu\~noz, \emph{Kink dynamics in the $\phi^4$ model: asymptotic stability for odd perturbations in the energy space}, J. Amer. Math. Soc. 30 (2017), 769--798.



\bibitem{KMM1}\bysame, \emph{Nonexistence of small, odd breathers for a class of nonlinear wave equations}, Letters in Mathematical Physics, May 2017, Volume 107, Issue 5, pp 921--931.

\bibitem{KL}  Krieger, Joachim; Lindblad, Hans, \emph{On stability of the catenoid under vanishing mean curvature flow on Minkowski space}. Dyn. Partial Differ. Equ. 9 (2012), no. 2, 89--119.

\bibitem{Lindblad} H. Lindblad,  \emph{A remark on global existence for small initial data of the minimal surface equation in Minkowskian space time},  Proc. Amer. Math. Soc. 132 (2004), no. 4, 1095--1102.

\bibitem{MM}
Y. Martel and F. Merle, \emph{A {L}iouville theorem for the critical
  generalized {K}orteweg-de {V}ries equation}, J. Math. Pures Appl. (9)
  \textbf{79} (2000), no.~4, 339--425.

\bibitem{MM1}
Y. Martel and F. Merle, \emph{Asymptotic stability of solitons for subcritical generalized {K}d{V}
  equations}, Arch. Ration. Mech. Anal. \textbf{157} (2001), no.~3, 219--254.
  
{
\bibitem{MM2}
Y. Martel and F. Merle, \emph{Asymptotic stability of solitons for subcritical  gKdV equations revisited}.
Nonlinearity, 18 (2005), no.~1, 55-80. }

\bibitem{MMT} Y. Martel, F. Merle, and T.-P. Tsai, \emph{Stability in $H^1$ of the sum of $K$ solitary waves for some nonlinear Schr\"odinger equations}, Duke Math. J. Volume 133, Number 3 (2006), 405--466.

\bibitem{MR}
F. Merle and P. Rapha\"el, \emph{The blow-up dynamic and upper bound on
  the blow-up rate for critical nonlinear {S}chr\"odinger equation}, Ann. of
  Math. (2) \textbf{161} (2005), no.~1, 157--222.

\bibitem{MPP} C. Mu\~noz, F. Poblete, and J. C. Pozo, \emph{Scattering in the energy space for Boussinesq equations}, preprint 2017 arXiv:1707.02616.


\bibitem{NS} W. Neves, and D. Serre, \emph{Ill-posedness of the Cauchy problem for totally degenerate system of conservation laws}, EJDE, Vol. 2005 (2005), No. 124, pp. 1--25.

\bibitem{Sch} Schr\"odinger, E. (1943). \emph{A new exact solution in non- linear optics (two-wave system)}. Proc. Roy. Irish Acad. A 49: 59--66.

\bibitem{Ste} Atanas Stefanov, \emph{Global regularity for the minimal surface equation in Minkowskian geometry}, Forum Math. 23 (2011), 757--789.

\bibitem{W} Whitham, G. B. (1974). \emph{Linear and nonlinear waves}, New York, Wiley.

\end{thebibliography}
\end{document}